\documentclass[reqno,12pt]{amsart}
\usepackage{amsmath,amssymb,latexsym}	
\newcommand{\V}{\mathrm{V}}
\newcommand{\W}{\mathrm{W}}
\newcommand{\G}{\mathrm{G}}

\newcommand{\GL}{\mathrm{GL}}
\newcommand{\SL}{\mathrm{SL}}

\newcommand{\R}{\mathbb{R}}
\newcommand{\Q}{\mathbb{Q}}
\newcommand{\Z}{\mathbb{Z}}
\newcommand{\N}{\mathbb{N}}
\newcommand{\C}{\mathbb{C}}

\newtheorem{theorem}{Theorem}[section]

\newtheorem{lemma}[theorem]{Lemma}
\newtheorem{corollary}[theorem]{Corollary}

\author{Sandip Singh}
\address{School Of Mathematics, Tata Institute of Fundamental Research, Homi Bhabha Road, Mumbai 400005, India}
\email{sandips@math.tifr.res.in}

\newenvironment{remark}[1][Remark]{\begin{trivlist}
\item[\hskip \labelsep {\bfseries #1.}]}{\end{trivlist}}

\begin{document}

\title{Coxeter groups are not higher rank arithmetic groups}
\vskip 5mm

\begin{abstract}
 Let $\W$ be an irreducible finitely generated Coxeter group. The geometric representation of $\W$ in $\GL(\V)$ provides a discrete embedding in the orthogonal group of the Tits form (the associated bilinear form of the Coxeter group). If the Tits form of the Coxeter group is non-positive and non-degenerate, the Coxeter group does not contain any finite index subgroup isomorphic to an irreducible lattice in a semisimple group of $\R$-rank $\geq2$.
\end{abstract}
\maketitle

\section{Introduction}
Let $S=\{s_1,s_2,\ldots,s_n\}$ be a finite set and $\W$ be a group generated by $S$ with the relations $$(s_{i}s_{j})^{m_{i,j}}=1,$$ where $m_{i,i}=1,\ \forall\ 1\leq{i}\leq{n}$ and $m_{i,j}\in\{2,3,\ldots,\infty\},\ \forall i\neq{j}$. The group $\W$ is called the Coxeter group. The Coxeter system ($\W,S$) is called irreducible if the Coxeter graph (\cite[Section 2.1]{Hum}) is connected. Now we define a symmetric bilinear form (Tits form) $B$ on a vector space $\V$ of dim $n$ over $\R$, with a basis $\{e_1,e_2,\ldots,e_n\}$ in one-to-one correspondence with $S$ as $$B(e_i,e_j)=-\mbox{cos}\left(\frac{\pi}{m_{i,j}}\right),\ \forall\ 1\leq{i,j}\leq{n}.$$ (This expression is interpreted to be -1 in case $m_{i,j}=\infty$.)

For each $s_i\in{S}$, we can now define a reflection $\sigma_i:\V\rightarrow{\V}$ by the rule: $$\sigma_i\lambda=\lambda{-} 2 B(e_i,\lambda)e_i.$$ Clearly $\sigma_{i}e_i = -e_i$, while $\sigma_i$ fixes $H_i=\{v\in{\V}| B(v,e_i)=0\}$ pointwise. In particular, we see that $\sigma_i$ has order 2 in $\GL(\V)$. The bilinear form $B$ is preserved by all of the elements $\sigma_i$, and hence it will be preserved by each element of the subgroup of $\GL(\V)$ generated by the $\sigma_i(1\leq{i}\leq{n}).$

By defining $s_i\mapsto\sigma_i$, we get a unique homomorphism $\sigma: \W\rightarrow\GL(\V)$ sending $s_i$ to $\sigma_i$, and the group $\sigma(\W)$ preserves the form $B$ on $\V$; and for each pair $s_i,s_j\in{S}$, the order of $s_{i}s_{j}$ in $\W$ is precisely $m_{i,j}$ (\cite[Proposition 5.3]{Hum}). Also, the representation $\sigma: \W\rightarrow\GL(\V)$ is {\it faithful} (\cite[Corollary 5.4]{Hum}).

Relative to the basis $\{e_1,e_2,\ldots,e_n\}$ of $\V$, we can identify $\V$ with $\R^n$ and $\GL(\V)$ with $\GL(n,\R)$, the latter in turn being viewed as an open set in $\R^{n^2}$. It follows from  \cite[Proposition 6.2]{Hum} that $\sigma(\W)$ is a discrete subgroup of $\GL(\V)$.

In this paper we will assume that the Coxeter system $(\W,S)$ is {\it irreducible} and the Tits form $B$ is {\it non-degenerate} and the Coxeter group $\W$ is {\it infinite}. By the earlier observations, it follows that $\W$ is a discrete subgroup of the corresponding orthogonal group $\G:=\mathrm{O}(B)(\R)$. Moreover, $\G$ is a real Lie group, with a Haar measure which provides a notion of volume $\nu$ for $\W\backslash\G$, the homogeneous space of right cosets of $\G$ with respect to $\W$. If the measure $\nu$ on $\W\backslash\G$ is finite and $\G$-invariant, then $\W$ is  a lattice in $\G$.

The goal of this paper is to prove Theorem \ref{t1} (stated below), which has been proved in \cite{CLR} also, by using a different technique. They have proved that an infinite Coxeter group has a subgroup of finite index which admits a homomorphism onto $\Z$ (\cite[Theorem 1.1]{CLR}) and used it to prove the theorem. I have tried here to give an elementary proof of the theorem by using a Bourbaki exercise (\P 12, Exercise $\S4$  of Chapter \rm{V} in \cite{Bou})  and Margulis superrigidity (Theorem \ref{t4}, below). 

\begin{theorem}\label{t1}
 If $\W$ is an irreducible finitely generated Coxeter group with the non-positive and non-degenerate Tits form, then it does not contain any finite index subgroup isomorphic to an irreducible lattice in a connected semisimple Lie group without non-trivial compact factor groups, of real rank $\geq2$.
\end{theorem}

In fact more is true: 
\begin{theorem}\label{t2}
\noindent $(a)$ If $\W$ is an irreducible finitely generated Coxeter group with the non-positive and non-degenerate Tits form, then it does not contain any finite index subgroup isomorphic to a higher rank S-arithmetic group (i.e. lattice in a product of Lie groups and $p$-adic groups).

For example, the Coxeter group $\W$ does not contain any finite index subgroup isomorphic to $\SL_2(\Z[\frac{1}{p}])$ in $\SL_2(\R)\times\SL_2(\Q_p)$. 

\noindent $(b)$ More generally, if $k_1, k_2,\ldots, k_r$ are local fields and $\G_1, \G_2,\ldots$, $\G_r$ are semisimple algebraic groups defined over $k_1, k_2,\ldots, k_r$ respectively such that each $\G_i$ has $k_i$-rank $\geq{1}$ and $\sum_{i=1}^{r}k_i\mbox{-rank }(\G_i)\geq{2}$, then $\W$ does not contain any finite index subgroup isomorphic to an irreducible lattice $\Gamma$ in $\prod_{i=1}^{r}\G_i(k_i)$.

For example, the Coxeter group $\W$ does not contain any finite index subgroup isomorphic to $\SL_3(\mathbb{F}_p[t])$ in $\SL_3(\mathbb{F}_p((\frac{1}{t})))$.
\end{theorem}

Theorem \ref{t2} can be proved by the same method used for the proof of Theorem \ref{t1} using Theorem \ref{t3} (stated below) and the superrigidity of lattices in semisimple groups over local fields of arbitrary characteristic (see \cite{Mar}; cf. \cite{Ven}). Therefore, in this paper we will prove Theorem \ref{t1}; and for the sake of completeness of the proof we will also prove the following theorem (stated in \cite{Bou} as an exercise):

\begin{theorem}[\P 12, Exercise $\S$ 4 of Chapter \rm{V} in \cite{Bou}]\label{t3}
If $\W$ is a lattice in $\mathrm{O}(B)(\R)$, then $B$ has signature $(n-1,1)$ and $B(v,v)<0$, for all $v\in{\mathrm{C}}$, where $\mathrm{C}:=\{v\in{\V}| B(v,e_i) > 0, \forall\ 1\leq{i}\leq{n}\}$.
\end{theorem}

Note that a Coxeter group $\W$ can not be a lattice in $\mathrm{O}(B)(\R)$ $=\mathrm{O}(n-1,1)$, for $n>10$ (17, Exercise $\S$ 4 of Chapter \rm{V} in \cite{Bou}). 

To prove Theorem \ref{t1} we will use the following theorem of G. A. Margulis:

\begin{theorem}[Theorem 6.16 of Chapter \rm{IX} in \cite{Mar}]\label{t4}
Let $\mathrm{H}$ be a connected semisimple Lie group without nontrivial compact factor groups. Let $\Gamma\subset\mathrm{H}$ be a lattice, $k$ a local field, $\mathrm{F}$ a connected semisimple $k$-group, and $\delta:\Gamma\longrightarrow\mathrm{F}(k)$ a homomorphism such that the subgroup $\delta(\Gamma)$ is Zariski dense in $\mathrm{F}$. Assume that rank $\mathrm{H}\geq 2$ and the lattice $\Gamma$ is irreducible. Then,

\begin{itemize} \item[(a)] for $k$ isomorphic neither to $\R$ nor to $\C$, i.e. for non-archimedean $k$, the subgroup $\delta(\Gamma)$ is relatively compact in $\mathrm{F}(k)$.

\item[(b)]  for $k=\R$, if the group $\mathrm{F}$ is adjoint and has no nontrivial $\R-$ anisotropic factors, then $\delta$ extends, uniquely, to a continuous homomorphism $\tilde{\delta}:\mathrm{H}\longrightarrow\mathrm{F}(\R)$. 
\end{itemize}
\end{theorem}
In this paper (Section \ref{s3}), we will also show that a right-angled Coxeter group $\W$ generated by 3 elements is isomorphic to a lattice in the group $\mathrm{O}(B)(\R)=\mathrm{O}(2,1)$ of real rank 1.

\section{Proof of Theorem \ref{t3}}\label{s2}
The proof has been sketched in the Bourbaki exercise (\P 12, Exercise $\S$ 4 of Chapter \rm{V} in \cite{Bou}), and for the sake of completeness we fill in the details.

If $s_i\in S$, denote by $\mathrm{A_i}$, the set of $x\in\V$ such that $B(x,e_i)>0$. Clearly $\mathrm{C}=\cap_{i=1}^n\mathrm{A}_i$ is an open set in $\V$, if $S$ is finite. The following theorem is from \cite{Bou}:

\begin{theorem}[Tits]\label{t5}
If $w\in\W$ and $\mathrm{C}\cap{w}\mathrm{C}\neq\emptyset$, then $w=1$. 
\end{theorem}

Let $\G$ be a closed subgroup of $\GL(\V)$ containing $\W$. Let $\G$ be unimodular and $\mathrm{D}$ be a half line of $\V$ contained in $\mathrm{C}$ i.e. $\mathrm{D}=\R_{>0}v\subset\mathrm{C}$, for some $v\in\mathrm{C}$, and let $\G_{\mathrm{D}}$ be the stabilizer of $\mathrm{D}$ in $\G$. With these notation, we get the following lemma:
\begin{lemma}
 Let $\Delta$ be the set of elements $g\in\G$ such that $g(\mathrm{D})\subset\mathrm{C}$. Then $\Delta$ is open, stable under right multiplication by $\G_{\mathrm{D}}$, and that the composite map $\Delta\longrightarrow\G\longrightarrow\W\backslash\G$ is injective, where $\W\backslash\G$ denotes the homogeneous space of right cosets of $\G$ with respect to $\W$.
\end{lemma}
\begin{proof}
 First, we show that $\Delta$ is open in $\G$. For, $\Delta=\{g\in\G| g(v)\in\mathrm{C}\}$, where $v\in\V$ such that $\mathrm{D}=\R_{>0}v\subset\mathrm{C}$. We define a map $f:\G\longrightarrow\V$ by $g\mapsto{g(v)}$. It is clear that $f$ is continuous and $\mathrm{C}$ is open in $\V$, hence $f^{-1}(\mathrm{C})=\Delta$ is open in $\G$.
 
Now we show that $\Delta$ is stable under right multiplication by $\G_{\mathrm{D}}$. For, let $h\in\G_{\mathrm{D}}$ and $g\in\Delta$. Then $$gh(v)=g(\alpha{v})=\alpha g(v)\in\mathrm{C},\ \mbox{for some}\ \alpha\in\R_{>0},$$ and this shows that $gh\in\Delta$.

Finally, we show that the composite map $\Delta\longrightarrow\G\longrightarrow\W\backslash\G$ is injective. For, let $g_1,g_2\in\Delta$ such that $\W{g_1}=\W{g_2}$ i.e. ${g_1g_2^{-1}}\in\W$. Since $g_2(\mathrm{D})\subset\mathrm{C}$, $\mathrm{D}\subset{g_2^{-1}(\mathrm{C})}$. That is, $g_1(\mathrm{D})\subset{g_1g_2^{-1}(\mathrm{C})}$. Also, $g_1(\mathrm{D})\subset{\mathrm{C}}$, therefore $g_1g_2^{-1}(\mathrm{C})\cap\mathrm{C}\neq\emptyset$. Hence by Theorem \ref{t5}, we get $ g_1g_2^{-1}=1$. This shows that the composite map $\Delta\longrightarrow\G\longrightarrow\W\backslash\G$ is injective.
\end{proof}

\begin{lemma}
 Let $\mu$ be a Haar measure on $\G$. If $\mu(\Delta)$ is finite, the subgroup $\G_{\mathrm{D}}$ is compact.
\end{lemma}

\begin{proof}
 Since $\Delta$ is an open set containing the identity element of $\G$ and the group $\G$ is locally compact, there exists a compact neighbourhood $\mathrm{K}$ of the identity element contained in $\Delta$. 
 
We now claim that there exist finitely many elements $h_i\in\G_{\mathrm{D}}$ such that every set of the form $\mathrm{K}h$, with $h\in\G_{\mathrm{D}}$, meets one of the $\mathrm{K}h_i$. For, suppose on the contrary that $\forall k\in\N$ and $\mathcal{H}_k=\{h_1, h_2,\ldots, h_k\}$ collection of elements in $\G_\mathrm{D}$, there exists $h_{k+1}\in\G_\mathrm{D}$ such that $\mathrm{K}h_{k+1}\cap\left(\cup_{i=1}^{k}\mathrm{K}h_{i}\right)=\emptyset$. It is also clear that $\mathrm{K}h_i\cap\mathrm{K}h_j=\emptyset,\ \forall i\neq{j}$. 

Since $\Delta$ is stable under right multiplication by any element of $\G_{\mathrm{D}}$, we get $\mathrm{K}h\subset\Delta,\ \forall h\in\G_{\mathrm{D}}$. Hence 
 $$\mu(\Delta)\geq\mu(\cup_{i=1}^{\infty}\mathrm{K}h_i)=\sum_{i=1}^{\infty}\mu(\mathrm{K}h_i)=\sum_{i=1}^{\infty}\mu(\mathrm{K})=\infty $$
 (since $\G$ is unimodular and $\mathrm{K}$ contains an open subset of $\G$, $\mu(\mathrm{K})>0$), which is a contradiction to the given hypothesis.  Therefore, $\exists$ $\mathcal{H}_r=\{h_1,h_2,\ldots,h_r\}$ a finite collection of elements in $\G_{\mathrm{D}}$ such that $\forall h \in\G_{\mathrm{D}}, \ \mathrm{K}h\cap\mathcal{K}h_i\neq\emptyset, \mbox{for some}\ i\in\{1,2,\ldots,r\}$, which shows that $\G_{\mathrm{D}}\subset\cup_{i=1}^{r}\mathrm{K}^{-1}\mathrm{K}h_i$ and hence $\G_{\mathrm{D}}$ is compact (since $\G_{\mathrm{D}}$ is a closed subset of $\G$ and  $\cup_{i=1}^{r}\mathrm{K}^{-1}\mathrm{K}h_i$ is compact).
 \end{proof}
 
\begin{lemma}\label{l1}
 Let $\nu$ be a non-zero positive measure on $\W\backslash\G$, invariant under $\G$. If $\nu(\W\backslash\G)<\infty$, then $\G_{\mathrm{D}}$ is compact.
\end{lemma}

\begin{proof}
Recall that $\G$ is unimodular with a Haar measure $\mu$ and $\nu$ is a non-zero positive measure on $\W\backslash\G$, invariant under $\G$. Let $\nu'$ be a Haar measure on $\W$. Since $\W$ is a discrete subgroup of $\GL(V)$, $\nu'$ is actually the counting measure (up to a scalar multiple) on $\W$. We prove here that $\mu(\Delta)<\infty$, which proves that $\G_{\mathrm{D}}$ is compact, using the last lemma.

We have a relation in $\mu, \nu$ and $\nu'$ as
  \begin{equation}\label{e1}
   \int_{\G}f\,\mathrm{d}\mu= \int_{\W\backslash\G}\left(\int_{\W}f(wg)\,\mathrm{d}\nu'(w)\right)\,\mathrm{d}\nu(\W{g}),\ \forall f\in\mathcal{C}_{c}(\G)
  \end{equation}
where $\mathcal{C}_{c}(\G)$ is the space of all compactly supported continuous functions on $\G$.

Let the symbol $f\prec\Delta$ means that $f\in\mathcal{C}_{c}(\G)$ with $0\leq{f}\leq{1}$ and the support of $f$ is contained in $\Delta$.  Since $\Delta$ is open in $\G$, we get
\begin{equation}\label{e2}
 \mu(\Delta)=\mbox{sup}\left\{\int_{\G}f\,\mathrm{d}\mu: f\prec\Delta\right\}.
\end{equation}
Let $f\prec\Delta$. By (\ref{e1}), we get
\begin{align}\label{e3}
   \int_{\G}f\,\mathrm{d}\mu &= \int_{\W\backslash\G}\left(\int_{\W}f(wg)\,\mathrm{d}\nu'(w)\right)\,\mathrm{d}\nu(\W{g})\nonumber\\
&\leq \int_{\W\backslash\G}\left(\int_{\W}\chi_{\Delta}(wg)\,\mathrm{d}\nu'(w)\right)\,\mathrm{d}\nu(\W{g})
  \end{align}
where $\chi_{\Delta}$ is the characteristic function of $\Delta$. 

Since $wg\in\Delta\Leftrightarrow w\in\Delta g^{-1}$, we get 
\begin{align}\label{e4}
 \int_{\W\backslash\G}\left(\int_{\W}\chi_{\Delta}(wg)\,\mathrm{d}\nu'(w)\right)\,\mathrm{d}\nu(\W{g}) &=\int_{\W\backslash\G}\nu'(\Delta g^{-1}\cap\W)\,\mathrm{d}\nu(\W g)\nonumber\\
&=\int_{\W\backslash\G}\#(\Delta g^{-1}\cap\W)\,\mathrm{d}\nu(\W g)
\end{align}
where $\#(\Delta g^{-1}\cap\W)$ denotes the number of elements in the set $\Delta g^{-1}\cap\W$.

Since $x\in\Delta g^{-1}\cap\W\Leftrightarrow xg\in\Delta$  and $x\in\W$, we get $$xg(\mathrm{D})\subset\mathrm{C}\ \mbox{i.e.}\ xg(v)\in\mathrm{C},\ \forall x\in\Delta g^{-1}\cap\W\ (\because\ \mathrm{D}=\R_{>0}v).$$ 
Now we claim that $\#(\Delta g^{-1}\cap\W)\leq 1$. For, let $x_1, x_2\in\Delta g^{-1}\cap\W$. Then, we get the following: $$x_1g(v)=c_1\in\mathrm{C} \quad \mbox{and}\quad x_2g(v)=c_2\in\mathrm{C}$$
\begin{align*}
 &\Rightarrow x_2 x_1^{-1}(c_1)=x_2x_1^{-1}(x_1(gv))=x_2(gv)=c_2\\
&\Rightarrow x_2x_1^{-1}(\mathrm{C})\cap\mathrm{C}\neq\emptyset \\
&\Rightarrow x_2x_1^{-1}=1 \qquad (\mbox{by Theorem \ref{t5}})\\
&\Rightarrow x_2 = x_1.
\end{align*}
Therefore, $\#(\Delta g^{-1}\cap\W)\leq 1$, and we get
\begin{align}\label{e5}
 \int_{\W\backslash\G}\#(\Delta g^{-1}\cap\W)\,\mathrm{d}\nu(\W g) &\leq\int_{\W\backslash\G}\,\mathrm{d}\nu(\W g)\nonumber\\
&=\nu(\W\backslash\G).
\end{align}
By (\ref{e3}), (\ref{e4}) and (\ref{e5}), we get $$\int_{\G}f\,\mathrm{d}\mu\leq\nu(\W\backslash\G).$$ As $f\prec\Delta$ was chosen arbitrarily, we get $\mu(\Delta)\leq\nu(\W\backslash\G)$ (by using (\ref{e2})), and hence $\mu(\Delta)<\infty$.
\end{proof}
Now we prove Theorem \ref{t3} using the above lemmas. We have $B$, a non-degenerate bilinear form on $\V$. Let $\G$ be the group of real points of the orthogonal group of $B$ and $\mu$ be a Haar measure  on $\G$. It is clear that the group $\G$ is unimodular and contains $\W$.  Since $\W$ is infinite, the bilinear form $B$ is not positive definite and it has the signature $(p,q)$, where $p+q=n$ and $p,q\geq 1$. We prove few more lemmas to prove Theorem \ref{t3}.
\begin{lemma}
 $B(v,v)\neq 0$, for some $v\in\mathrm{C}$.
\end{lemma}
\begin{proof}
  Since for any $v\in\mathrm{C}, \mathrm{C}-v$ is an open subset of $\V$ containing the origin 0 (since $\mathrm{C}$ is an open subset of $\V$), $\V$ is generated by $\mathrm{C}-v$ (as an abelian group). In particular, $\mathrm{C}-v$ generates $\V$ as a vector space over $\R$, therefore there exists $\{v_1-v,v_2-v,\ldots,v_n-v\}$ a basis of $\V$ over $\R$ contained in $\mathrm{C}-v$, where $v_i\in\mathrm{C}, \forall\ 1\leq i\leq n$. Now if possible, let $B(v,v)=0,\ \forall v\in\mathrm{C}$.
\begin{align}\label{e6}
\Rightarrow B(u,v)&=\frac{1}{2}(B(u+v,u+v)-B(u,u)-B(v,v))\nonumber\\
&=0\qquad \forall u,v\in\mathrm{C}\quad(\because\ \forall u,v\in\mathrm{C},\ u+v\in\mathrm{C}). \end{align}
Now we show that if $B(v,v)=0,\ \forall v\in\mathrm{C}$, then $B\equiv 0$, which gives a contradiction (since $B$ is non-zero). Since $v_i, v\in\mathrm{C}$, using the bilinearity of $B$ and (\ref{e6}), we get 
$$B(v_i-v,v_j-v)=0,\ \forall\ 1\leq i,j\leq n$$ i.e. $B\equiv 0$. Therefore $\exists v\in\mathrm{C}$ such that $B(v,v)\neq 0$.
\end{proof}
 Let $v\in\mathrm{C}$ be an element for which $B(v,v)\neq 0$. Let $\mathrm{L}_v=\{u\in\V| B(u,v)=0\}$. Since $B(v,v)\neq 0$, $V=\R v\oplus\mathrm{L}_v$. Now take $\mathrm{D}=\R_{>0}v\subset\mathrm{C}$, a half line contained in $\mathrm{C}$. We have a basis $\{v,u_1,u_2,\ldots,u_{n-1}\}$ of $\V$ over $\R$, where $\{u_1,u_2,\ldots,u_{n-1}\}$ is a basis of $\mathrm{L}_v$ over $\R$. With respect to this basis of $\V$, $B=B_1\oplus B_2$, where $B_1=B|_{\R v}$ and $B_2=B|_{\mathrm{L}_v}$. The symmetric matrix associated to the bilinear form $B$, with respect to this basis, is of the form
 $$B=\begin{pmatrix}
 B_1(v,v) &0 &0 &\ldots &0\\
0 \\
0 &B_2\\
\vdots \\
0
\end{pmatrix}.$$

The group $\G=\mathrm{O}(B)(\R)\leq\GL(n,\R)$, is unimodular with a Haar measure $\mu$ and it contains the Coxeter group $\W$ as a discrete subgroup. Let $\nu$ be a $\G$-invariant measure on the quotient $\W\backslash\G$ such that $\nu\left(\W\backslash\G\right)<\infty$ i.e. $\W$ is a lattice in $\G$.

 Let $\mathrm{H}=\mathrm{O}(B_2)(\R)\leq\GL(\mathrm{L}_v)$ be the orthogonal group of the bilinear form $B_2$ on $\mathrm{L}_v$. It is clear that $$\G'=\left\{\begin{pmatrix}
                                                                                                                                                                   1 &0 &\ldots &0\\
0 \\
\vdots &\quad &h\\
0                                                                                                                                                                  \end{pmatrix}: h\in\mathrm{H}\right\}$$ is a closed subgroup of $\G$ and $\forall g\in\G',\ g(v)=v$ i.e. $\G'$ is a closed subgroup of $\G_{\mathrm{D}}$, therefore it is compact (by Lemma \ref{l1}). 

Also, $\G'$ is isomorphic (as a Lie group) to $\mathrm{H}=\mathrm{O}(B_2)(\R)$, therefore $\mathrm{H}$ is a compact subgroup of $\GL(\mathrm{L}_v)$. It shows that the bilinear form $B_2$ is either positive definite or negative definite. Since the group $\W$ is infinite, the bilinear form $B$ can not be positive or negative definite. Therefore $B$ has the signature $(n-1,1)\ \mbox{or}\ (1,n-1)$.

Now we show that $B$ can not have the signature $(1,n-1)$.
\begin{lemma}
 If there is a relation $(s_is_j)^{m_{i,j}}=1$, for some $i\neq j$ and $2\leq m_{i,j}<\infty$ in the generators of the Coxeter group $\W$ and the bilinear form $B$ as above, then $B$ has the signature $(n-1,1)$.
\end{lemma}
\begin{proof}
 For $2\leq m_{i,j}<\infty$, $B(e_i,e_j)=-\mbox{cos}\left(\frac{\pi}{m_{i,j}}\right)>-1$, and hence
\begin{align*}
 B(\lambda e_i+\delta e_j,\lambda e_i+\delta e_j)&=\lambda^2 B(e_i,e_i)+\delta^2 B(e_j,e_j)+2\lambda\delta B(e_i,e_j)\\
&=\lambda^2+\delta^2+2\lambda\delta B(e_i,e_j)\\
&>\lambda^2+\delta^2-2\lambda\delta=(\lambda-\delta)^2\geq 0
\end{align*}
(since $B(e_i,e_j)>-1$).
Therefore $\forall \lambda,\delta\in\R$ and $(\lambda,\delta)\neq (0,0), B(\lambda e_i+\delta e_j,\lambda e_i+\delta e_j)>0$. 

Let $\V_{i,j}=\R e_i\oplus\R e_j$ be a subspace of $\V$. The restriction of the bilinear form $B$ on $\V_{i,j}$ is non-degenerate and positive definite. Therefore $\V=\V_{i,j}\oplus\V_{i,j}^{\perp}$, and with respect to a basis of $\V$ which is the union of a basis of $\V_{i,j}$ and a basis of $\V_{i,j}^{\perp}$, the matrix of the bilinear form $B$ is 
$$\begin{pmatrix}
    1 & 0 & 0 & \ldots & 0\\
0 & 1 & 0 & \ldots & 0\\
0 & 0 \\
\vdots & \vdots & \quad & B|_{\V_{i,j}^{\perp}}\\
0 & 0
   \end{pmatrix}$$
and $B|_{\V_{i,j}^{\perp}}$ is non-degenerate.

 The above matrix form of the bilinear form $B$ shows that its signature is $(p,q)$, where $p,q\in\N,\ p+q=n, \mbox{and} \ p\geq 2$. Therefore the possibility for the signature of $B$ to be $(1,n-1)$ is excluded i.e. $B$ has the signature $(n-1,1)$.
\end{proof}
\begin{lemma}
 If $(s_is_j)^{\infty}=1$, for $i\neq j$ and $s_is_i=1$, $\forall i,j\in\{1,2,\ldots,n\}$ are the only relations in the generators of the Coxeter group $\W$ and the bilinear form $B$ as above, then $B$ has the signature $(n-1,1)$.
\end{lemma}
\begin{proof}
These relations mean that all the vertices in the Coxeter graph of the Coxeter group $\W$ are joined by an edge of weight $\infty$, and $B(e_i,e_i)=1$, and $B(e_i,e_j)=-1$, for $i\neq j$. These relations are not possible in a Coxeter group $\W$ with 2 generators ($\because$ $B$ is non-degenerate), therefore to have the possibility stated in the statement of the lemma, $n$ must be $\geq 3$. 

Since all the vertices are joined by an edge in the Coxeter graph, the Coxeter graph  contains a triangle. Let $s_1, s_2$ and $s_3$ be any three vertices which are joined to each other to form a triangle. Let $\V_1=\R e_1\oplus\R e_2\oplus\R e_3$ be a subspace of $\V$, and $B_1=B|_{\V_1}$ be a bilinear form on $\V_1$. Now we show that $B_1$ has the signature $(2,1)$, which shows that $\V=\V_1\oplus\V_1^{\perp}$ and hence the signature of $B$ is $(p,q)$ with $p\geq 2$. 

The matrix form of $B_1$ with respect to the basis $\{e_1,e_2,e_3\}$ of $\V_1$ over $\R$ is
$$B_1=\begin{pmatrix}
       1 & -1 & -1\\
-1 & 1 & -1\\
-1 & -1 & 1
      \end{pmatrix}.$$ One can check easily that $2, 2, -1$ are the eigenvalues of the matrix $B_1$. 
      
Since a symmetric matrix is orthogonally diagonalizable, the signature of the bilinear form $B_1$ is $(2,1)$. It shows that the possibility for the signature of the bilinear form $B$ to be $(1,n-1)$ is excluded. Therefore the signature of the bilinear form $B$ is $(n-1,1)$.
\end{proof}

Since we had $\V=\R v\oplus\mathrm{L}_v$, where $v\in\mathrm{C}$ is an element for which $B(v,v)\neq 0$, and $\mathrm{L}_v=\{u\in\V| B(u,v)=0\}$, the condition on the signature of $B$ forces $B(v,v)<0$ (since $B|_{\mathrm{L}_v}$ is positive definite and $B$ is non-degenerate and non-positive). The above proof also shows that if $B(u,u)\neq 0$, then $B(u,u)<0$,  for any $u\in\mathrm{C}$. 

Now we show that $B(u,u)\neq 0$, for any $u\in\mathrm{C}$. Otherwise, $\exists{u}\in\mathrm{C}$ such that $B(u,u)=0$. Since the bilinear form $B$ is non-degenerate, $\exists{u'}\in\V$ such that $B(u',u')=0$ and $B(u,u')=1$ (see \cite[Theorem 6.10]{Jac}). Also, for any $\alpha,\beta>0$ in $\R$, $B(\alpha{u}+\beta{u'},\alpha{u}+\beta{u'})=2\alpha\beta>0$. Since $u\in\mathrm{C}$, and $\mathrm{C}$ is open in $\V$, $\exists\alpha,\beta>0$ in $\R$ such that $\alpha{u}+\beta{u'}\in\mathrm{C}$ and $B(\alpha{u}+\beta{u'},\alpha{u}+\beta{u'})=2\alpha\beta>0$, which is a contradiction. Therefore $B(u,u)\neq0$, $\forall u\in\mathrm{C}$. Hence $B(u,u)<0$, $\forall u\in\mathrm{C}$; and it completes the proof of Theorem \ref{t3}. \qed

\section{Proof of Theorem \ref{t1}}

Let $\mathrm{O}(B)$ be the orthogonal group  of the bilinear form $B$ and $\mathrm{O}(p,q)$ be the group of real points of the group $\mathrm{O}(B)$ i.e. $\mathrm{O}(p,q)=\mathrm{O}(B)(\R),$ where $(p,q)$ is the signature of $B$ with $p,q\geq 1$, and $ p+q=n$. Let $\mathrm{SO}(B)$ be the connected component of the identity element of $\mathrm{O}(B)$, and $\mathrm{SO}(p,q)=\mathrm{SO}(B)(\R).$ The subgroup $\mathrm{SO}(p,q)$ has finite index (four) in the group $\mathrm{O}(p,q),$ therefore any finite index subgroup $\mathrm{L}'$ of the Coxeter group $\W$ contains a finite index subgroup $\mathrm{L}\leq\mathrm{SO}(p,q)$, namely $\mathrm{L}=\mathrm{L}'\cap\mathrm{SO}(p,q)$. If $\mathrm{L}'$ is isomorphic to an irreducible lattice $\Gamma'$ in a semisimple group $\mathrm{H}$ of $\R$-rank $\geq2$, then $\mathrm{L}$ will be isomorphic to a finite index subgroup $\Gamma$ of $\Gamma'$. Also, it can be shown easily that a finite index subgroup $\Gamma$ of an irreducible lattice $\Gamma'$ is an irreducible lattice in 
$\mathrm{H}$. We prove some lemmas which will be used in the proof of Theorem \ref{t1}. 

\begin{lemma}\label{l4}
 There exists a connected semisimple adjoint group $\tilde{\G}$ and an (central) isogeny $\pi:\mathrm{SO}(B)\longrightarrow\tilde{\G}$. 
\end{lemma}
For a proof, see  \cite[Theorem 2.6]{P-R}.

In fact, $\tilde{\G}$ is an $\R$-simple group (since the group $\mathrm{SO}(B)$ has maximal normal subgroup $\{\pm\mathrm{I}\}$ which is the center of $\mathrm{SO}(B)$ and $\pi$ is central therefore the kernel of $\pi$ is $\{\pm\mathrm{I}\}$).
\begin{lemma}\label{l2}
 If $\mathrm{L}$ is a discrete subgroup of $\mathrm{SO}(B)(\R)=\mathrm{SO}(p,q)$, then $\pi(\mathrm{L})$ is a discrete subgroup of $\tilde{\G}(\R)$.
\end{lemma}
\begin{proof}
 The homomorphism $\pi$ is an open map and its kernel is finite. Now using the discreteness of $\mathrm{L}$, it can be shown easily that $\pi(\mathrm{L})$ is a discrete subgroup of $\tilde{\G}(\R)$.
\end{proof}
\begin{lemma}\label{l3}
If $\mathrm{L}$ is a Zariski dense subgroup of $\mathrm{SO}(B)$, then $\pi(\mathrm{L})$ is a Zariski dense subgroup of $\tilde{\G}$.
\end{lemma}
\begin{proof}
 Since the map $\pi:\mathrm{SO}(B)\longrightarrow\tilde{\G}$ is continuous with respect to the Zariski topology, we get $\pi(\overline{\mathrm{L}})\subseteq\overline{\pi(\mathrm{L})}$. Therefore $\overline{\pi(\mathrm{L})}=\tilde{\G}$ (since $\overline{\mathrm{L}}=\mathrm{SO}(B)$).
\end{proof}
\begin{lemma}\label{l5}
 $\R$-rank $(\mathrm{SO}(B))$ = $\R$-rank $(\tilde{\G})$.
\end{lemma}
\begin{proof}
We will show that if $\mathrm{T}$ is an $\R$-split torus in $\mathrm{SO}(B)$, then $\pi(\mathrm{T})$ is an $\R$-split torus in $\tilde{\G}$. For, let $\mathrm{T}$ be an $\R$-split torus in $\mathrm{SO}(B)$ i.e.  all the characters $\chi: \mathrm{T}\longrightarrow\mathbb{G}_m$ are defined over $\R$. It is clear that $\pi(\mathrm{T})$ is a connected, abelian subgroup of $\tilde{\G}$. Also, $\pi(\mathrm{T})$ is diagonalizable over $\C$ (since under a homomorphism of algebraic groups, torus maps to a torus). 

To show that $\pi(\mathrm{T})$ is $\R$-split, it is enough to show that all the characters $\chi: \pi(\mathrm{T})\longrightarrow\mathbb{G}_m$ are defined over $\R$. For, let us define $\chi': \mathrm{T}\longrightarrow\mathbb{G}_m$ as $\chi'(t)=\chi(\pi(t))$. It is clear that $\chi'$ is a character of the torus $\mathrm{T}$ which is $\R$-split, therefore $\chi'$ is defined over $\R$. Now we show that $\chi$ is fixed under the action of $\mbox{Gal}(\C/\R)$ on $\mbox{Hom}(\mathrm{T},\mathbb{G}_m)$. For, let $\sigma\in\mbox{Gal}(\C/\R)$. We have 
\begin{align*}
 \chi(\pi(t))&=\chi'(t)\\
&= (\sigma.\chi')(t)\\
&= \sigma(\chi'(\sigma^{-1}t))\\
&= \sigma(\chi\circ\pi(\sigma^{-1}t))\\
&= \sigma(\chi\sigma^{-1}(\sigma.\pi)(t)) \\
&= (\sigma.\chi)(\pi(t)) \hspace{3.0cm} (\because \chi'\mbox{ and }\pi \mbox{ are defined over }\R).\end{align*}
 Since the above equality is true for all $t\in\mathrm{T}$ and $\pi$ is surjective, therefore we get $\sigma.\chi=\chi$, for all $\sigma\in\mbox{Gal}(\C/\R)$. Hence all the characters $\chi: \pi(\mathrm{T})\longrightarrow\mathbb{G}_m$ are defined over $\R$ i.e. $\pi(\mathrm{T})$ is an $\R$-split torus in $\tilde{\G}$. Since $\pi$ has finite kernel, we get $\R$-rank($\tilde{\G}$) = $\R$-rank($\mathrm{SO}(B)$).
\end{proof}

\begin{theorem}\label{t6}
 Let $\mathrm{L}$ be a discrete subgroup of the group $\mathrm{SO}(p,q)$. Let $\mathrm{H}$ be a connected semisimple Lie group without non-trivial compact factor groups, of real rank $\geq2$ with trivial center. Let $\Gamma\leq\mathrm{H}$ be an irreducible lattice and $\delta:\Gamma\longrightarrow\mathrm{L}\leq\mathrm{SO}(B)(\R)=\mathrm{SO}(p,q)$ be an isomorphism and $\delta(\Gamma)=\mathrm{L}$ is Zariski dense in $\mathrm{SO}(B)$. Let $\tilde{\G}$ be a connected semisimple adjoint group with an (central) isogeny $\pi:\mathrm{SO}(B)\longrightarrow\tilde{\G}$. Let $\delta':\Gamma\longrightarrow\pi(\mathrm{L})\leq\tilde{\G}(\R)$ be a continuous homomorphism defined as $\delta'=\pi\circ\delta$. Let $\tilde{\G}$ has no nontrivial $\R$-anisotropic factors and $\tilde{\G}(\R)^\circ$ be the connected component of the identity element in $\tilde{G}(\R)$. Then $\delta'$ extends uniquely to an isomorphism $\tilde{\delta'}:\mathrm{H}\longrightarrow\tilde{\G}(\R)^{\circ}$, and the group $\tilde{\G}(\R)$ has $\R$-rank $\
geq 2$, and $\pi(\mathrm{L})$ is a lattice in $\tilde{\G}(\R)$.
\end{theorem}
\begin{proof}
 The group $\tilde{\G}$ is adjoint, and has no nontrivial $\R$-anisotropic factors and $\pi(\mathrm{L})$ is a discrete subgroup of $\tilde{\G}(\R)$ (by Lemma \ref{l2}), and it is also Zariski dense in $\tilde{\G}$ (by Lemma \ref{l3}). Therefore by Theorem \ref{t4} we get a continuous homomorphism $\tilde{\delta'}:\mathrm{H}\longrightarrow\tilde{\G}(\R)$ with $\tilde{\delta'}|_{\Gamma}=\delta'$. Since the group $\tilde\delta'(\mathrm{H})$ is a connected semisimple group which is Zariski dense in $\tilde{\G}$ (since $\tilde\delta'(\Gamma)=\pi(\mathrm{L})$ is Zariski dense in $\tilde{\G}$), it follows from \cite{Mar} (Remark 6.17 (ii) of Chapter IX) that $\tilde{\delta'}(\mathrm{H})=\tilde{\G}(\R)^{\circ}$. Since $\mathrm{H}$ has trivial center and no nontrivial compact factor groups, $\Gamma$ is an irreducible lattice in $\mathrm{H}$, and $\delta'(\Gamma)=\pi(\mathrm{L})$ is a nontrivial discrete subgroup of $\tilde{\G}(\R)$, therefore it follows from \cite{Mar} (Remark 6.17 (iii) of Chapter IX) that $\tilde{\
delta'}$ is an isomorphism of $\mathrm{H}$ onto $\tilde{\G}(\R)^{\circ}$, and hence $\pi(\mathrm{L})$ is a lattice in $\tilde{\G}(\R)^{\circ}$, and the $\R$-rank of $\tilde{\G}(\R)$ is $\geq 2$. Since $\tilde{\G}(\R)^{\circ}$ is a finite index subgroup of $\tilde{\G}(\R)$, $\pi(\mathrm{L})$ is a lattice in $\tilde{\G}(\R)$.
\end{proof}

\begin{remark}
In the proof of Theorem \ref{t6}, the fact that $\mathrm{H}$ has trivial center, has been used only to show that $\tilde{\delta'}$ is an isomorphism. If the group $\mathrm{H}$ does not have trivial center, then the homomorphism $\tilde{\delta'}$ has finite kernel, and $\tilde{\delta'}(\Gamma)=\pi(\mathrm{L})$ is still a lattice in $\tilde{\G}(\R)$ (since under such homomorphism $\tilde{\delta'}$, a lattice maps onto a lattice). Therefore Theorem \ref{t6} is also true for a connected semisimple Lie group with non-trivial center, and without non-trivial compact factor groups, of real rank $\geq2$. 
\end{remark}

\begin{lemma}\label{l6}
 Let $\mathrm{L}$ be a discrete subgroup of $\mathrm{SO}(p,q)$, and $\tilde{\G},\pi$ as in Lemma \ref{l4}. If $\pi(\mathrm{L})$ is a lattice in $\tilde{\G}(\R)$, then $\mathrm{L}$ is a lattice in $\mathrm{SO}(p,q)$.
\end{lemma}
\begin{proof}
 Since $\mathrm{L}$ is a discrete subgroup of $\mathrm{SO}(p,q)$ and $\mathrm{SO}(p,q)$ is unimodular, the quotient $\mathrm{L}\backslash\mathrm{SO}(p,q)$ has an $\mathrm{SO}(p,q)$-invariant measure $\mu$. The homomorphism $\pi:\mathrm{SO}(p,q)\longrightarrow\tilde{\G}(\R)$ induces a continuous map $\tilde{\pi}:\mathrm{L}\backslash\mathrm{SO}(p,q)\longrightarrow\pi(\mathrm{L})\backslash\tilde{\G}(\R)$, which is defined as $\tilde{\pi}(\mathrm{L}g)=\pi(\mathrm{L})\pi(g)$. It can be checked easily that the pushforward measure $\tilde{\pi}_{*}(\mu)$ on the quotient $\pi(\mathrm{L})\backslash\tilde{\G}(\R)$ defined as $\tilde{\pi}_{*}(\mu)(\tilde{\mathrm{E}})=\mu(\tilde{\pi}^{-1}(\tilde{\mathrm{E}}))$, for all measurable subsets $\tilde{\mathrm{E}}$ of $\pi(\mathrm{L})\backslash\tilde{\G}(\R)$, is $\tilde{\G}(\R)$-invariant (since $\tilde{\pi}$ is surjective and $\mu$ is $\mathrm{SO}(p,q)$-invariant). Therefore by the uniqueness of a $\tilde{\G}(\R)$-invariant measure on the quotient $\pi(\mathrm{L})\backslash\
tilde{\G}(\R)$, we get $\tilde{\pi}_{*}(\mu)(\pi(\mathrm{L})\backslash\tilde{\G}(\R))<\infty$ (since $\pi(\mathrm{L})$ is a lattice in $\tilde{\G}(\R)$), and hence $\mu(\mathrm{L}\backslash\mathrm{SO}(p,q))<\infty$ i.e. $\mathrm{L}$ is a lattice in $\mathrm{SO}(p,q)$. 
\end{proof}

\begin{theorem}\label{t7}
The Coxeter group $\W$ is Zariski dense in the group $\mathrm{O}(B)$.
\end{theorem}
 For a proof, see \cite{B-H}.
\begin{lemma}\label{l7}
 Let $\G$ be a topological group and $\mathrm{L'},\mathrm{L}$ are subgroups of $\G$ such that $\mathrm{L}$ has finite index in $\mathrm{L'}$. Then $(\bar{\mathrm{L'}})^{o}=(\bar{\mathrm{L}})^{o}$, where $(\bar{\mathrm{L}})^{o}$ is the connected component of the identity element of the closure of $\mathrm{L}$ in $\G$.
\end{lemma}
\begin{proof}
  Since $\mathrm{L}$ has finite index $d$ (say) in $\mathrm{L'}$, 
\begin{align*}
&\qquad\mathrm{L'}=\cup_{i=1}^{d}\gamma_i\mathrm{L};\ \gamma_i\in\mathrm{L'}\\
&\Rightarrow \bar{\mathrm{L'}}=\cup_{i=1}^{d}\gamma_i\bar{\mathrm{L}};\ \gamma_i\in\mathrm{L'}\\
&\Rightarrow [\bar{\mathrm{L'}}:\bar{\mathrm{L}}]\leq d\\
&\Rightarrow \bar{\mathrm{L}} \mbox{ is a finite index subgroup of the group }\bar{\mathrm{L'}}.
\end{align*}
Hence $\bar{\mathrm{L}}$ is closed and open in $\bar{\mathrm{L'}}$ and $(\bar{\mathrm{L'}})^{o}\supset(\bar{\mathrm{L}})^{o}$, therefore $(\bar{\mathrm{L}})^{o}$ is open and closed in $(\bar{\mathrm{L'}})^{o}$ which is connected. This shows that $(\bar{\mathrm{L'}})^{o}=(\bar{\mathrm{L}})^{o}$.
\end{proof}

\begin{corollary}\label{c1}
 In the above lemma if we take $\G=\mathrm{O}(p,q)=\mathrm{O}(B)(\R)$, and $\mathrm{L'}=\W,$ the Coxeter group and $\mathrm{L}\leq\mathrm{SO}(p,q)\cap\W$ such that $[\W:\mathrm{L}]<\infty$, then $\bar{\mathrm{L}}=\mathrm{SO}(p,q)$ i.e. $\mathrm{L}$ is Zariski dense in $\mathrm{SO}(p,q).$ Hence $\mathrm{L}$ is Zariski dense in $\mathrm{SO}(B)$ {\rm(}$\because$ $\mathrm{SO}(B)(\R)=\mathrm{SO}(p,q)$ is Zariski dense in $\mathrm{SO}(B)${\rm)}.
\end{corollary}
\begin{proof}
The proof follows from Theorem \ref{t7} and Lemma \ref{l7}.
\end{proof}
\begin{lemma}\label{l8}
 If $\mathrm{L}$ is a lattice in $\mathrm{SO}(p,q)$, then $\mathrm{L}$ is also a lattice in $\mathrm{O}(p,q)$.
\end{lemma}
\begin{proof}
 Since $\mathrm{O}(p,q)$ is unimodular and $\mathrm{L}$ is a discrete subgroup of $\mathrm{O}(p,q)$, we get $\mathrm{L}\backslash\mathrm{O}(p,q)$ has a non-zero $\mathrm{O}(p,q)$-invariant measure $\mu$. Since  $\mathrm{SO}(p,q)$ is open in $\mathrm{O}(p,q)$, its Borel $\sigma$-algebra is a subalgebra of the Borel $\sigma$-algebra of $\mathrm{O}(p,q)$ and the restriction of $\mu$ on $\mathrm{L}\backslash\mathrm{SO}(p,q)$ is a non-zero $\mathrm{SO}(p,q)$-invariant measure. Now we claim that $\mu(\mathrm{L}\backslash\mathrm{O}(p,q))<\infty$. For, $$\mathrm{L}\backslash\mathrm{O}(p,q)=\{\mathrm{L}g|g\in\mathrm{O}(p,q)\},$$ and $$\mathrm{O}(p,q)=\{\mathrm{SO}(p,q)g_i|g_i\in\mathrm{O}(p,q),1\leq i\leq 4\},$$ i.e. $\forall g\in\mathrm{O}(p,q),\ \exists h\in\mathrm{SO}(p,q)\ \mbox{such that}\ g=hg_i$, for some $1\leq i\leq 4$. Therefore, $\mathrm{L}g=\mathrm{L}hg_i\in(\mathrm{L}\backslash\mathrm{SO}(p,q))g_i$, and $$\mathrm{L}\backslash\mathrm{O}(p,q)=\cup_{i=1}^{4}(\mathrm{L}\backslash\mathrm{SO}(p,q))g_i.$$ 
\begin{align*}\Rightarrow \mu(\mathrm{L}\backslash\mathrm{O}(p,q))&\leq\sum_{i=1}^{4}\mu((\mathrm{L}\backslash\mathrm{SO}(p,q))g_i)\\ &=\sum_{i=1}^{4}\mu(\mathrm{L}\backslash\mathrm{SO}(p,q))\\ &<\infty.
\end{align*}
It shows that $\mathrm{L}$ is a lattice in $\mathrm{O}(p,q)$.
\end{proof}
From the remark at the beginning of this section and Corollary \ref{c1}, it follows that if the Coxeter group $\W$ contains a finite index subgroup $\mathrm{L}\leq\mathrm{SO}(p,q)$, which is isomorphic to an irreducible lattice in a connected semisimple Lie group $\mathrm{H}$ without nontrivial compact factor groups, of real rank $\geq2$, then $\mathrm{SO}(p,q)$ has real rank $\geq 2$ (by Lemma \ref{l5} and Theorem \ref{t6}) i.e. $p,q\geq 2$, and $\mathrm{L}$ is a lattice in $\mathrm{SO}(p,q)$ (by Theorem \ref{t6} and Lemma \ref{l6}). Moreover, Lemma \ref{l8} shows that $\mathrm{L}$ is a lattice in $\mathrm{O}(p,q)$ also, and hence $\W$ becomes a lattice in $\mathrm{O}(p,q)$ (since a discrete subgroup $\W$ of a Lie group $\G$ which contains a lattice $\mathrm{L}$, is a lattice in $\G$). This is a contradiction to Theorem \ref{t3}, which has been proved in Section \ref{s2}; and it completes the proof of Theorem \ref{t1}. \qed

\section{Right-angled Coxeter group with three generators}\label{s3}
In this section we will do some computations and show that a right-angled Coxeter group $\W$ generated by 3 elements is isomorphic to a lattice in the group $\mathrm{O}(B)(\R)=\mathrm{O}(2,1)$ of real rank 1. 

Let $\W$ be the right-angled Coxeter group generated by 3 elements $s_1, s_2$ and $s_3$ with the relations: $(s_is_j)^{m_{i,j}}=1$, where $m_{i,i}=1, \forall i\in\{1,2,3\}$,  and $m_{1,2}=m_{2,3}=\infty$, $m_{1,3}=2$. Let $\R^3$ be a 3-dimensional vector space over $\R$ with a basis $\{e_1, e_2, e_3\}$. We define a symmetric bilinear form $B$ on $\R^3$ as $$B(e_i,e_j)=-\mbox{cos}\left(\frac{\pi}{m_{i,j}}\right),\quad\mbox{for} \ m_{i,j}\neq\infty,$$ and for $m_{i,j}=\infty$, we define $B(e_i,e_j)=-1$. With respect to the basis $\{e_1, e_2, e_3\}$, the matrix of $B$ is
$$B=\begin{pmatrix}
     1 & -1 & 0\\
-1 & 1& -1\\
0 & -1 & 1
    \end{pmatrix}.$$
One can check that the bilinear form $B$ is non-degenerate. 

Now we define a representation $\rho:\W\longrightarrow\GL(\R^3)$ by defining $\rho(s_i)(e_j)=e_j-2B(e_j,e_i)e_i$, which is faithful (by \cite[Corollary 5.4]{Hum}). It can be checked easily that $\rho$ maps the group $\W$ inside the orthogonal group $\mathrm{O}(B)(\R)$ of the bilinear form $B$. We will show  that the group $\W$ is mapped (by $\rho$) onto a finite index subgroup of $\mathrm{O}(B)(\Z)$, the group of integral points of the orthogonal group $\mathrm{O}(B)$, and it shows that the group $\W$ is a lattice in $\mathrm{O}(B)(\R)$. 

With respect to the basis $\{e_1, e_2, e_3\}$, the matrices of $\rho(s_1)$, $\rho(s_2)$ and $\rho(s_3)$ are
$$\rho(s_1)=\begin{pmatrix}
             -1 & 2 & 0\\
0 & 1 & 0\\
0 & 0 & 1
            \end{pmatrix},
\rho(s_2)=\begin{pmatrix}
             1 & 0 & 0\\
2 & -1 & 2\\
0 & 0 & 1
            \end{pmatrix},
\ \rho(s_3)=\begin{pmatrix}
                 1 & 0 & 0\\
0 & 1 & 0\\
0 & 2 & -1
                \end{pmatrix}.
$$
If we do some integral change in the basis of $\R^3$ over $\R$, and take $\{e_1+e_2, e_2, e_2+e_3\}$ as a basis of $\R^3$, then the corresponding matrices of $ \rho(s_1),\rho(s_2), \rho(s_3) \ \mbox{and}\ B,$ become
$$\rho(s_1)=\begin{pmatrix}
 1 & 2 & 2\\
0 & -1 & -2\\
0 & 0 & 1
\end{pmatrix},\rho(s_2)=\begin{pmatrix}
             1 & 0 & 0\\
0 & -1 & 0\\
0 & 0 & 1
            \end{pmatrix},\rho(s_3)=\begin{pmatrix}
                 1 & 0 & 0\\
-2 & -1 & 0\\
2 & 2 & 1
                \end{pmatrix}$$ and 
$$B=\begin{pmatrix}
     0 & 0 & -1\\
0 & 1 & 0\\
-1 & 0  & 0
    \end{pmatrix}.$$
It is now clear that the signature of the bilinear form $B$ is $(2,1)$.

The adjoint representation of $\SL(2,\R)$ on its Lie algebra $\mathrm{sl}(2,\R)$, maps the group $\mathrm{PSL}(2,\R)=\SL(2,\R)/\{\pm\mathrm{I}\}$ isomorphically onto its image and it preserves the killing form $K$ defined on $\mathrm{sl}(2,\R)$. The Lie algebra $\mathrm{sl}(2,\R)$ can be identified with $\R^3$ as a vector space over $\R$, with the basis $$\left\{e_1=\begin{pmatrix}
0 & 1\\
0 & 0                                                                                                     \end{pmatrix},e_2=\begin{pmatrix}
1 & 0\\
0 & -1                                                                                                     \end{pmatrix},e_3=\begin{pmatrix}
0 & 0\\
1 & 0                                                                                                     \end{pmatrix}\right\}.$$
 The killing form $K$ on $\mathrm{sl}(2,\R))$ is defined by
$$K(X,Y)=\frac{1}{2}\mbox{tr}(XY),\quad \forall \quad X,Y\in\mathrm{sl}(2,\R).$$ 
If we do some integral change in the basis of $\mathrm{sl}(2,\R)$ over $\R$ and take 
$$\left\{\epsilon_1=-2e_1=\begin{pmatrix}
0 & -2\\
0 & 0                                                                                                     \end{pmatrix},\epsilon_2=e_2=\begin{pmatrix}
1 & 0\\
0 & -1                                                                                                     \end{pmatrix},\epsilon_3=e_3=\begin{pmatrix}
0 & 0\\
1 & 0                                                                                                     \end{pmatrix}\right\}$$ as a basis of $\mathrm{sl}(2,\R)$ over $\R$, then the matrix of $K$ becomes
$$K=\begin{pmatrix}
     0 & 0 & -1\\
0 & 1 & 0\\
-1 & 0  & 0
    \end{pmatrix}.$$ Therefore the bilinear form $B$ associated to the Coxeter group $\W$, is equivalent to the killing form $K$ on $\mathrm{sl}(2,\R)$ over $\Z$, and the signature of $K$ is also $(2,1)$. Hence the group $\SL(2,\R)/\{\pm\mathrm{I}\}$ maps into $\mathrm{O}(2,1)\leq\GL(3,\R)$, by the adjoint representation $\mathrm{Ad}$ of $\SL(2,\R)$ on its Lie algebra, where $\mathrm{O}(2,1)=\mathrm{O}(B)(\R)$. Since the group $\SL(2,\R)/\{\pm\mathrm{I}\}$ is connected, it is mapped inside $\mathrm{SO}(2,1)$,  the connected component of the identity in $\mathrm{O}(2,1)$. In fact, $\mathrm{Ad}(\SL(2,\R)/\{\pm\mathrm{I}\})$=$\mathrm{SO}(2,1)$ ($\because$ dim $\SL(2,\R)/\{\pm\mathrm{I}\}$=dim $\mathrm{SO}(2,1)$) i.e. $\SL(2,\R)/\{\pm\mathrm{I}\}\cong\mathrm{SO}(2,1)$. Hence $\SL(2,\Z)/\{\pm\mathrm{I}\}$ is a lattice in $\mathrm{SO}(2,1)$. In fact, $\SL(2,\Z)/\{\pm\mathrm{I}\}$ is a lattice in  $\mathrm{O}(2,1)$ ($\because$ $\mathrm{SO}(2,1)$ has finite index in $\mathrm{O}(2,1)$).

The right-angled Coxeter group $\W$ is mapped inside $\mathrm{O}(B)(\Z)=\mathrm{O}(2,1)(\Z)$, by the representation $\rho$. We  construct a finite index subgroup $\mathrm{H}$ of $\SL(2,\Z)/\{\pm\mathrm{I}\}$ which preserves a lattice $\mathrm{L}$ in $\mathrm{sl}(2,\R)=\R^3$(as a vector space) i.e. $\mathrm{H}$ is also mapped inside $\mathrm{O}(2,1)(\Z)$, by the representation $\mathrm{Ad}$, and being a finite index subgroup of $\SL(2,\Z)/\{\pm\mathrm{I}\}$,  $\mathrm{H}$ becomes a lattice in  $\mathrm{O}(2,1)$. Also, we construct a finite index subgroup $\mathrm{H'}$ of $\W$ which is mapped onto $\mathrm{Ad}(\mathrm{H})$, by the representation $\rho$, and hence $\rho(\mathrm{H'})$ becomes a lattice in $\mathrm{O}(2,1)$, and $\W$ becomes a finite index subgroup of $\mathrm{O}(2,1)(\Z)$ i.e. a lattice in $\mathrm{O}(2,1)$.
\begin{lemma}
 The group $\SL(2,\Z)/\{\pm\mathrm{I}\}$ is generated by $w=\begin{pmatrix}
                                                      0 & -1\\
1 & 0
                                                     \end{pmatrix}$ and $x=\begin{pmatrix}
1 & 1\\
0 & 1
\end{pmatrix}$, and it has a presentation as $<w,x;w^2,(wx)^3>$ i.e. it is the free product of the cyclic group of order 2 generated by $w$ and the cyclic group of order 3 generated by $wx$.
\end{lemma}
For a proof, see Theorem 2 and the preceding remark of Chapter VII in \cite{Ser}.

We get $x^2=\begin{pmatrix}
             1 & 2\\
0 & 1
            \end{pmatrix}\mbox{ and }  wx^2w^{-1}=\begin{pmatrix}
1 & 0\\
-2 & 1
\end{pmatrix}=\begin{pmatrix}
1 & 0\\
2 & 1
\end{pmatrix}^{-1}.$

Let $\mathrm{H}$ be the subgroup of $\SL(2,\Z)/\{\pm\mathrm{I}\}$ generated by $\{x^2, wx^2w^{-1}\}.$ It can be shown using the  presentation of $\SL(2,\Z)/\{\pm\mathrm{I}\}$ as in the above lemma, that the subgroup $\mathrm{H}$ has finite index in $\SL(2,\Z)/\{\pm\mathrm{I}\}$. Also, one can show easily that $\mathrm{H}$ preserves the lattice $$\mathrm{L}=\Z\begin{pmatrix}
0 & -2\\
0 & 0 \end{pmatrix}\oplus\Z\begin{pmatrix}
1 & 0\\
0 & -1  \end{pmatrix}\oplus\Z\begin{pmatrix}
0 & 0\\
1 & 0  \end{pmatrix}$$ in $\mathrm{sl}(2,\R)$. Hence $\mathrm{H}$ is mapped inside $\mathrm{O}(2,1)(\Z)$, by the adjoint representation $\mathrm{Ad}$, and being a lattice ($\because$ it has finite index in $\SL(2,\Z)/\{\pm\mathrm{I}\}$) in $\mathrm{O}(2,1)(\R)$, it has finite index in $\mathrm{O}(2,1)(\Z)$. By an easy computation, we find that the matrices of $\mathrm{Ad}(x^2)$, $\mathrm{Ad}(wx^2w^{-1})^{-1}$ in $\mathrm{O}(2,1)(\R)$ with respect to the basis $$\left\{\epsilon_1=-2e_1=\begin{pmatrix}
0 & -2\\
0 & 0                                                                                                     \end{pmatrix},\epsilon_2=e_2=\begin{pmatrix}
1 & 0\\
0 & -1                                                                                                     \end{pmatrix},\epsilon_3=e_3=\begin{pmatrix}
0 & 0\\
1 & 0                                                                                                     \end{pmatrix}\right\},$$ are 
\begin{align}\label{e7}
\mathrm{Ad}(x^2)=\begin{pmatrix}
                                                 1 & 2 & 2\\
0 & 1 & 2\\
0 & 0 & 1
                                                \end{pmatrix},\quad \mathrm{Ad}(wx^2w^{-1})^{-1}=\begin{pmatrix}
                                                 1 & 0 & 0\\
4 & 1 & 0\\
8 & 4 & 1
                                                \end{pmatrix}.\end{align}
Let $\mathrm{H'}$ be the subgroup of the Coxeter group $\W$ generated by the set $\{s_2s_1, s_2s_3\}.$ It can be shown easily that the subgroup $\mathrm{H'}$ has finite index in the group $\W$. We find that the matrices of $\rho(s_2s_1)$ and $\rho(s_2s_3)$ in $\mathrm{O}(2,1)(\R)$ with respect to the basis $\{e_1+e_2, e_2, e_2+e_3\}$, are 
\begin{align}\label{e8}
 \rho(s_2s_1)=\begin{pmatrix}
 1 & 2 & 2\\
0 & 1 & 2\\
0 & 0 & 1
\end{pmatrix}\quad\mbox{ and }\quad \rho(s_2s_3)=\begin{pmatrix}
             1 & 0 & 0\\
2 & 1 & 0\\
2 & 2 & 1
            \end{pmatrix}.
\end{align}
Also, $\rho(s_2s_3)^2=\begin{pmatrix}
             1 & 0 & 0\\
4 & 1 & 0\\
8 & 4 & 1
            \end{pmatrix}=\mathrm{Ad}(wx^2w^{-1})^{-1}$, and hence by (\ref{e7}) and (\ref{e8}), we see that $\mathrm{H}$ is a subgroup of $\mathrm{H'}$. Therefore $\mathrm{H'}$ is a finite index subgroup of $\mathrm{O}(2,1)(\Z)$, and hence the Coxeter group $\W$ is also a finite index subgroup of $\mathrm{O}(2,1)(\Z)$ i.e. $\W$ is a lattice in $\mathrm{O}(2,1)$.

\section*{Acknowledgements}

I sincerely thank my advisor Professor T. N. Venkataramana for suggesting me this problem and for many useful discussions. I also thank Professor James E. Humphreys for suggesting me a reference to a Bourbaki (\cite{Bou}) exercise (Theorem \ref{t3} in this paper).

\end{document}